\documentclass[12pt]{amsart}
\usepackage[utf8]{inputenc}
\usepackage{amssymb, eurosym}
\usepackage{amsmath}
\usepackage{amsthm}

\usepackage[left=2cm,right=2cm,top=2.5cm,bottom=2.5cm]{geometry}

\usepackage{amsfonts}
 
\usepackage[all]{xy} 

\theoremstyle{definition}
\newtheorem{defi}{Definition}[section]
\newtheorem{lem}[defi]{Lemma}
\newtheorem{teo}[defi]{Theorem}
\newtheorem*{nnteo}{Main Theorem}

\newtheorem{prop}[defi]{Proposition}
\newtheorem{ex}[defi]{Example}

\theoremstyle{remark}
\newtheorem{obs}[defi]{Remark}

\DeclareMathOperator{\rad}{\operatorname{rad}}

\DeclareMathOperator{\Ga}{\Gamma}
\DeclareMathOperator{\Rf}{\mathfrak{R}}
\DeclareMathOperator{\Rr}{\mathcal{R}}

\DeclareMathOperator{\Sc}{\mathcal{S}} 

\DeclareMathOperator{\lp}{\operatorname{lb}}

\title[Gröbner bases and irreducible morphisms]{Gröbner bases for mesh relations and applications to compositions of irreducible morphisms}

\author[Chust]{Viktor Chust}
\address{(Viktor Chust) Institute of Mathematics and Statistics - University of São Paulo, São Paulo, Brazil}
\thanks{corresponding author: viktorchust.math@gmail.com}

\author[Coelho]{Flávio U. Coelho}
\address{(Flávio U. Coelho) Institute of Mathematics and Statistics - University of São Paulo, São Paulo, Brazil}

\date{}

\subjclass[2020]{Primary 16G70, Secondary 16G10, 13P10}

\keywords{irreducible morphisms, compositions of irreducible morphisms, mesh category, Gröbner bases}

\begin{document}

\maketitle

\begin{abstract}
We give a necessary condition for the existence of a path of $n$ irreducible morphisms between indecomposable modules whose composition lies in the ($n+1$)-power of the radical. In order to do that, we consider the general criterion given by C. Chaio, P. Le Meur and S. Trepode, which relates these compositions with zero paths in the mesh category, and then study morphisms in the mesh category by providing Gröbner bases for the subspaces generated by the mesh relations.
    
\end{abstract}

\section*{Introduction}

The theory of Gröbner bases was introduced by B. Buchberger (whose advisor was W. Gröbner) in his thesis \cite{Buch} in order to study certain problems regarding ideals in polynomial rings over a field. Namely, the key feature of this theory is an algorithm (called the {\it Buchberger algorithm}) which yields a set of generators for such an ideal, called a \textit{Gröbner basis}, and allows effective computations of the solutions for these problems. After its initial inception, applications and generalizations of Gröbner bases have been discovered across multiple areas in Mathematics, such as ideals in free associative algebras (in connection with Bergman's Diamond Lemma - \cite{Berg}), syzygies and module theory over more general rings, and even far-reaching ones, such as the coloring problem in graph theory - see \cite{AdLous}, e.g., for these last two. We also recommend \cite{AdLous, CLS} for a proper introduction to Gröbner bases theory.

We also highlight the use of this theory in \cite{FFG,Gre,LM}, which, in the context of Representation Theory of Algebras, study bound path algebras over quivers by obtaining Gröbner bases for the ideals generated by relations.

Inspired by this last application, here we study Gröbner bases for the spaces of relations which the define the so-called \textit{mesh categories}. This has allowed us to obtain new results about the problem in Representation Theory of compositions of irreducible morphisms, which we now explain. Here we study the representation theory of finitely generated, right $A$-modules, where $A$ stands for a finite-dimensional algebra over an algebraically closed field $k$. We will assume the reader is familiar with basic concepts of Representation Theory and Auslander-Reiten theory. For unexplained concepts of these theories, we recommend \cite{AC, AC2, ARS}.

If $X,Y$ are two indecomposable $A$-modules, then a {\it radical morphism} $X \rightarrow Y$ is simply a non-isomorphism between these two modules. Let $\rad_A$ denote the \textit{radical ideal} of the module category, that is, the ideal formed by all radical morphisms between $A$-modules. We also denote by $\rad_A^n$ the $n$-th power of the ideal $\rad_A$, and by $\rad_A^{\infty}$ the intersection of the $\rad_A^n$ for all $n \in \mathbb{N}$. Then Auslander-Reiten theory defines the \textit{irreducible morphisms} between two $A$-modules: a morphism $f:X \rightarrow Y$ between two indecomposable $A$-modules is irreducible if and only if it belongs to $\rad_A(X,Y) \setminus \rad_A^2(X,Y)$.

However, what would be a similar pattern for more than one irreducible morphism does not happen: although we can guarantee that the composite of $n$ irreducible morphisms between indecomposable modules belongs to $\rad^n$, if $n>1$ it is not true in general that it does not belong to $\rad^{n+1}$. By the way, \cite{CCT1} gives an example of two irreducible morphisms whose composite belongs to $\rad^{\infty}$. Studying when there might be non-trivial composites of $n$ irreducible morphisms lying in $\rad^{n+1}$ has become a general problem.

Although a classical result by K. Igusa and G. Todorov (\cite{IT1}) rules out composites of $n$ irreducible morphisms in $\rad^{n+1}$ along a sectional path, this problem began to be directly addressed in \cite{CCT1}, where the authors gave necessary and sufficient conditions for the existence of a non-zero composite of 2 irreducible morphisms in $\rad^3$.

After \cite{CCT1}, many other works introduced results about the problem of composites. (In \cite{CCs}, we have included a list of the main sources in that subject). Most of these articles study this problem in particular cases, e. g., for small values of $n$, for specific paths along which the composite is taken or by imposing conditions on the Auslander-Reiten component.

A general criterion that holds for all values of $n$ was given in \cite{CMT1}, Proposition 5.1, and later extended to perfect fields in \cite{CMT2}, Proposition 3 by the same authors C. Chaio, P. Le Meur and S. Trepode. By using the techniques of \textit{coverings of quivers}, \textit{mesh categories} and Riedtmann's \textit{well-behaved functors} (which were originally introduced in \cite{Rie, BG}), they have conditioned the existence of compositions of $n$ irreducible morphisms in $\rad^{n+1}$ to certain zero paths of morphisms in the mesh category.

In this article, we build on this idea, studying zero paths in the mesh category as we obtain \textit{Gröbner bases} for the spaces of morphisms identified with zero in the mesh category. This in particular has led us into the discovery of a new necessary condition for the existence of $n$ irreducible morphisms between indecomposable modules whose composite belongs to $\rad^{n+1}$:

\begin{nnteo}
\label{th:main}
Let $\Gamma$ be a component of the Auslander-Reiten quiver of $A$, and let $X,Y$ be two indecomposable modules of $\Gamma$. Then the following hold:

\begin{enumerate}

\item Suppose that for every non-projective module $Z$ from $\Gamma$ such that there are paths $X \rightsquigarrow \tau Z$ and $Z \rightsquigarrow Y$, the Auslander-Reiten sequence ending in $Z$ has a middle term with at least two indecomposable summands. Then, if $X=X_0 \xrightarrow{h_1} X_1 \xrightarrow{h_2} \ldots \xrightarrow{h_n} X_n = Y$ is a path of irreducible morphisms, it holds that $h_n \ldots h_1 \notin \rad^{n+1}(X,Y)$.

\item Let $X \xrightarrow{h_1} \ldots \xrightarrow{h_n} Y$ be a path of irreducible morphisms in $\Gamma$ such that $h_n \ldots h_1 \in \rad^{n+1}(X,Y)$. Then there are a non-projective module $Z$ in $\Gamma$ and paths $\gamma_1: X \rightsquigarrow \tau Z$ and $\gamma_2: Z \rightsquigarrow Y$ over $\Gamma$ such that:
    
    \begin{enumerate}
        \item If $0 \rightarrow \tau Z \rightarrow E \rightarrow Z \rightarrow 0$ is the Auslander-Reiten sequence ending at $Z$, then $E$ is indecomposable.

        \item The path $\xymatrix{X \ar@{~>}[r]^{\gamma_1} & \tau Z \ar[r] & E \ar[r] & Z \ar@{~>}[r]^{\gamma_2} & Y}$  thus obtained is \textit{homotopic (see definition in \S~\ref{subsec:homotopy})} to the path given by the $h_i$'s.
    \end{enumerate}

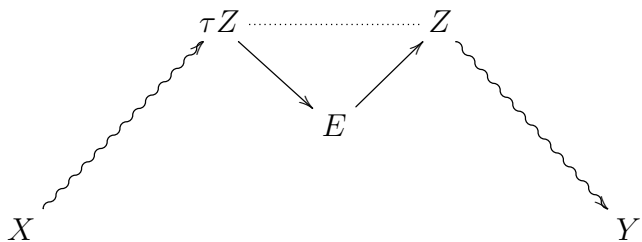
    \begin{displaymath}
        \xymatrix{ &&\tau Z \ar[dr] \ar@{.}[rr]&& Z \ar@{~>}[ddrr] && \\
        &&& E \ar[ur] &&& \\
        X \ar@{~>}[uurr] &&&&&& Y}
    \end{displaymath}
    
\end{enumerate}
\end{nnteo}

So our goal here is to develop the techniques and results necessary to prove the theorem above. This article is organized as follows: in Section~\ref{sec:prelim} we establish definitions and notations regarding \textit{quivers}, \textit{translation quivers}, \textit{homotopy}, \textit{coverings} and \textit{mesh categories}. In Section~\ref{sec:some grobner}, we recall some ideas from Gröbner bases theory, while fixing notation. In Section~\ref{sec:mesh-canonical rel}, we introduce the concept of \textit{mesh-canonical relations} on morphisms in the mesh category. Then, in Section~\ref{sec:grobner bases for mesh}, we prove (Theorem~\ref{th:mesh canonicas eh grobner}) that, with a certain condition, mesh-canonical relations form a Gröbner basis for the space of morphisms identified with zero in the mesh category, relatively to a special order which we call \textit{mesh-lexicographic}. Section~\ref{sec:zero paths} is devoted to state and prove Theorem~\ref{th:his neq 0}, using this Gröbner basis to study the existence of combinations of paths vanishing in the mesh category. This will be the key step towards concluding the proof of our main theorem above on Section~\ref{sec:proof of main thm}.

\section{Preliminaries}
\label{sec:prelim}

\subsection{Quivers}
\label{subsec:quiv}

A \textbf{quiver} $Q$ is a quadruple $Q=(Q_0,Q_1,s,e)$, where $Q_0$ is the set of \textbf{vertices} of $Q$, $Q_1$ is the set of \textbf{arrows}, and $s,e:Q_1 \rightarrow Q_0$ are two functions which determine, respectively, the \textbf{start} and the \textbf{end} of an arrow. To each arrow $\alpha \in Q_1$, we can consider its formal inverse $\alpha^{-1}$, and we establish by convention that $s(\alpha^{-1}) = e(\alpha)$ and $e(\alpha^{-1}) = s(\alpha)$.

By a \textbf{walk} over $Q$ we mean a sequence $\beta_n \ldots \beta_1$, where, for every $1 \leq i \leq n$, $\beta_i$ is either an arrow or the inverse of an arrow, and for every $2 \leq i \leq n$, $s(\beta_i) = e(\beta_{i-1})$.
If $w= \beta_n \ldots \beta_1$ is a walk, we extend the notations used for start and end vertices: $s(w) \doteq s(\beta_1)$ and $e(w) \doteq e(\beta_n)$. If $w$ and $w'$ are two walks and $e(w) = s(w')$, we can define the composition $w'w$ naturally by juxtaposition.

If $\beta_n \ldots \beta_1$ is a walk and all $\beta_i's$ are arrows (and not inverses of arrows), then we say that this walk is a \textbf{path of length $n$}. Additionally, one associates to each vertex $x$ of $Q$ a trivial \textbf{path of length 0}, denoted by $\epsilon_x$. Of course, $s(\epsilon_x)=e(\epsilon_x) =x$.

If $x \rightsquigarrow y$ is a path from $x$ to $y$, then we say that $x$ is a \textbf{predecessor} of $y$ and that $y$ is a \textbf{successor} of $x$. If this path has length exactly 1 (i.e., it is an arrow), then we use the terms \textbf{immediate predecessor} and \textbf{immediate successor}. Given a vertex $z$, we denote by $z^{-}$ and $z^{+}$, respectively, its sets of immediate predecessors and immediate successors.

\subsection{Translation quivers}
\label{subsec:tr quivs}

A \textbf{translation quiver} is a quiver without loops $\Gamma$ equipped with a function (called \textit{translation}) $\tau: \operatorname{dom} \tau \subseteq \Gamma_0 \rightarrow \Gamma_0$ which is partially defined on the vertices of $\Gamma$. The vertices which do not belong to the domain of $\tau$ are called \textit{projective}, the vertices which fall outside the range of $\tau$ are called \textit{injective}. Also, $\Gamma$ and $\tau$ must satisfy that $\tau: x \mapsto \tau(x)$ is a bijection between non-projective and non-injective vertices, and also that, for each pair of vertices $x$ and $y$ of $\Gamma$, with $x$ non-projective, there is a bijection $\sigma: \alpha \mapsto \sigma(\alpha)$ between arrows of the form $y \rightarrow x$ and arrows of the form $\tau(x) \rightarrow y$. Also, we shall be always assuming that every translation quiver is \textit{locally finite}. That means that for every vertex $x$ of $\Gamma$, there is only a finite number of arrows starting at $x$ and a finite number ending at it. 

Note that Auslander-Reiten quivers, as well as their connected components, yield examples of locally finite translation quivers, with $\tau$ being the Auslander-Reiten translation.

In general, we will always use the letters $\tau$ e $\sigma$ to denote the functions associated to a given translation quiver $\Gamma$.

We still have to give some additional definitions regarding translation quivers. Let $\Gamma$ be a translation quiver. If $x_0 \xrightarrow{\alpha_1} x_1 \xrightarrow{\alpha_2} \ldots \xrightarrow{\alpha_n} x_n$ is a path over $\Gamma$, we call it a \textbf{sectional path} if $x_i \neq \tau x_{i+2}$ for every $i$ between $0$ and $n-2$. 

Also, given a non-projective vertex $x$ of $\Gamma$, the \textbf{mesh} ending at $x$ is the full subquiver of $\Gamma$ determined by all the arrows that end at $x$ and all the arrows that start at $\tau x$.

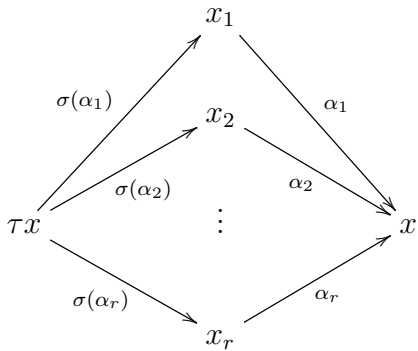
    \begin{displaymath}
        \xymatrix{ && x_1 \ar[ddrr]^{\alpha_1}&& \\
        && x_2 \ar[drr]_{\alpha_2}&&\\
        \tau x \ar[uurr]^{\sigma(\alpha_1)} \ar[urr]_{\sigma(\alpha_2)} \ar[drr]_{\sigma(\alpha_r)}&& \vdots && x \\
        &&x_r \ar[urr]_{\alpha_r} &&}
    \end{displaymath}

We introduce an auxiliary definition that will help us to state our main results to follow, which is that of \textit{open} and \textit{closed} paths. We have chosen these names with a very loose inspiration from the geometrical realization of a translation quiver introduced by \cite{BG}.

\begin{defi}
    Let $\Gamma$ be a translation quiver. We say that a path $x = x_0 \rightarrow x_1 \rightarrow \ldots \rightarrow x_m = y$ in $\Gamma$ is \textbf{closed} if for every time there is $0 \leq i \leq m-2$ such that $x_i = \tau x_{i+2}$, we have that $x_i$ has at least two immediate successors in $\Gamma$. If this does not happen, we say that the path is \textbf{open}.
\end{defi}

\subsection{Homotopy}
\label{subsec:homotopy}

Let $\Gamma$ be a translation quiver. Following \cite{BG, CMT1}, we are going to define the concept of \textit{homotopy} in this context. Consider the smallest equivalence relation $\sim$ between walks over $\Gamma$ which satisfies conditions a) to d) below:

\begin{enumerate}
    \item [a)] If $\alpha: x \rightarrow y$ is an arrow of $\Gamma$, then $\alpha \alpha^{-1} \sim \epsilon_y$ e $\alpha^{-1} \alpha \sim \epsilon_x$. (Here $\epsilon_x$ e $\epsilon_y$ denote the paths of length zero over the vertices $x$ and $y$ respectively).
    
    \item [b)] If $x$ is a non-projective vertex and $\alpha: z \rightarrow x$, $\alpha': z' \rightarrow x$ are two arrows ending at $x$, then $\alpha (\sigma \alpha) \sim \alpha' (\sigma \alpha')$.
    
    \item [c)] If $\gamma_1,\gamma,\gamma',\gamma_2$ are walks over $\Gamma$ such that $\gamma \sim \gamma'$ and the compositions $\gamma_1 \gamma \gamma_2$ and $\gamma_1 \gamma' \gamma_2$ are well-defined, then $\gamma_1 \gamma \gamma_2 \sim \gamma_1 \gamma' \gamma_2$.

    \item [d)] If $\alpha,\beta: x \rightarrow y$ are two arrows with the same starting and ending vertex, then $\alpha \sim \beta$.
\end{enumerate}

The relation $\sim$ is then called the \textbf{homotopy} between walks over $\Gamma$.

\subsection{Coverings}
\label{subsec:coverings}

Later on we will need the concept of {\it coverings of translation quivers} (\cite{Rie,BG,CMT1}). Let $\Delta$ and $\Gamma$ be two translation quivers. We say that a quiver morphism $\pi: \Delta \rightarrow \Gamma$ is a \textbf{covering} if: (1) A vertex $x$ of $\Delta$ is projective (or injective, respectively) if and only if $\pi(x)$ is also projective (or injective, respectively); (2) $\pi$ commutes with the translation functions in $\Delta$ and $\Gamma$ (wherever these are defined); and (3) For every vertex $x$ of $\Delta$, the map $\alpha \mapsto \pi(\alpha)$ induces a bijection from the set of arrows of $\Delta$ starting at $x$ (or ending at $x$, respectively) and the set of arrows of $\Gamma$ starting at $\pi(x)$ (or ending at $\pi(x)$, respectively). 

As one can see from this definition, a covering of a translation quiver $\Gamma$ is another translation quiver $\Delta$ which locally resembles $\Gamma$, much like the definition of covering spaces in Algebraic Topology.

\subsection{Mesh categories}
\label{subsec:mesh categories}

Let $\Delta$ be a translation quiver. Following \cite{Rie,BG,CMT1}, we are going to define the {\bf path category} $k\Delta$ and the {\bf mesh category} $k(\Delta)$ over $\Lambda$.

\begin{itemize}
    \item The objects of $k\Delta$ are the vertices of $\Delta$;

    \item Given two vertices $x,y \in \Delta_0$, the morphisms $x \rightarrow y$ are given by the formal $k$-linear combinations of the paths from $x$ to $y$ in $\Delta$;

    \item $k(\Delta)$ is the quotient of the category $k\Delta$ by the ideal generated by all the morphisms $m_x \doteq \sum_{\alpha \in \Delta_1,s(\alpha) = x} (\sigma^{-1}\alpha) \alpha$, with $x$ running through non-injective vertices of $\Delta$. We call $m_x$ as a \textbf{mesh relation} and the ideal which they generate as the \textbf{mesh ideal}.
\end{itemize}

For every $n \geq 1$,  define the ideal $\Rf^n k(\Delta)$ as the ideal of $k(\Delta)$ generated by the classes of paths of length $n$ over $\Delta$. The ideal $\Rf k(\Delta) \doteq \Rf^1 k(\Delta)$ is usually called the {\it radical of the mesh category}, although in general it is not precisely a radical in the category-theoretic sense.

\section{Some Gröbner bases theory}
\label{sec:some grobner}

In this section we shall define some concepts and fix some notations regarding Gröbner bases. For the reader who is unfamiliar with Gröbner bases or wants more details about the approach we follow here, we refer to the books \cite{AdLous, CLS}.

We have fixed $k$ a field. Consider a finite-dimensional space $V$ over $k$. Then choose \textbf{$k$-basis} $\mathcal{B} = \{v_1,\ldots,v_n\}$ of $V$, and choose a \textbf{total order} $v_1>v_2 > \ldots > v_n$ over the elements of $\mathcal{B}$.

Given $v \in V$, let $\alpha_1,\ldots,\alpha_n \in k$ be the only scalars such that $v = \alpha_1 v_1 + \ldots + \alpha_n v_n$. If $v \neq 0$, there is a least $i$ such that $\alpha_i \neq 0$, and we can define $\lp(v) = v_i$, which is the \textbf{leading basis element} in the decomposition of $v$.

\begin{defi}
    Let $f,g,r \in V$ with $g \neq 0$. Take $v_i = \lp(g)$, let $\alpha_i \in k$ be the coefficient of $v_i$ in $f$ and $\beta_i \in k$ be the coefficient of $v_i$ in $g$. Then we say that $f$ \textbf{simply reduces} to $r$ modulo $g$, and we write $f \xrightarrow{g} r$, if $\alpha_i \neq 0$ and it holds that $r = f - \frac{\alpha_i}{\beta_i} g$. We say that $f$ {\bf reduces} to $r$ modulo a finite set $G$ and denote $f \xrightarrow{G}_+ r$ if there is a finite sequence of reductions modulo elements in $G$ starting from $f$ and ending with $r$.

\end{defi}

An element $r \in V$ is called \textbf{reduced} relatively to a finite set $G \subseteq V \setminus \{0\}$ if $r=0$ or if $r$ cannot be reduced modulo $G$. 

If $f \xrightarrow{G}_+ r$ and $r$ is reduced relative to $G$, then we say that $r$ is a \textbf{remainder} for $f$ relatively to reduction modulo $G$.

\begin{teo}
    Let $W$ be a subspace of $V$. The following conditions are equivalent for a finite subset $G = \{g_1,\ldots,g_s\}$ of $W \setminus \{0\}$:

    \begin{enumerate}
        \item For every $f \in W \setminus \{0\}$ there is an index $i$ between 1 and $s$ such that $\lp(f) = \lp(g_i)$.

        \item It holds that $f \in W$ if and only if $f \xrightarrow{G}_+ 0$.

        \item For every $f \in V$, the remainders of possible reductions of $f$ modulo $G$ coincide.
    \end{enumerate}
\end{teo}

We say that $G$ is a {\bf Gröbner basis} for $W$  if it satisfies one of the equivalent conditions from the above theorem. It follows from item 3 above that if $G$ is a Gröbner basis for $W$, then the set $G$ generates $W$.

Once more we remark that the choice of a basis and the order of the elements in that basis are determinant: a set can be a Gröbner basis relatively to one such choice and not for another.

\begin{defi}
    If $G \subseteq V \setminus \{0\}$ is a finite set of non-zero vectors and $G$ is a Gröbner basis for $\langle G \rangle$, then we simply say that $G$ is a set that forms a \textbf{Gröbner basis}.
\end{defi}

\section{Mesh-canonical relations}
\label{sec:mesh-canonical rel}

We have seen that the mesh relations generate the mesh ideal, but we are interested in another set of generators, related to the first one, which is that of \textit{mesh-canonical relations}, which we now introduce.

Let $\Gamma$  be a translation quiver, and fix two vertices $x,y \in \Gamma_0$. Then $k\Ga(x,y)$ is the $k$-vector space whose basis are the paths from $x$ to $y$ over $\Ga$, and $k(\Ga)(x,y)$ is the quotient $k\Ga(x,y)/I(x,y)$, where $I(x,y)$ is a subspace generated by the mesh relations over $\Ga$. 

We say that an element $r \in I(x,y)$ is a \textbf{mesh-canonical relation} between $x$ and $y$ if $r$ has the form $r = \gamma_2 m_z \gamma_1$, where $z$ is a non-injective vertex of $\Ga$, $\gamma_1: x \rightsquigarrow z$ and $\gamma_2: \tau^{-1} z \rightsquigarrow y$ are paths over $\Gamma$. More explicitly, if $\alpha_1: z \rightarrow z_1, \ldots, \alpha_l: z \rightarrow z_l$ are all the arrows that start in $z$, then $m_z = \sum_{j = 1}^l (\sigma^{-1} \alpha_j)\alpha_j$ and

$$r = \gamma_2  \left(\sum_{j = 1}^l (\sigma^{-1} \alpha_j)\alpha_j\right) \gamma_1 = \sum_{j = 1}^l \gamma_ 2 (\sigma^{-1} \alpha_j)\alpha_j \gamma_1 $$

\begin{displaymath}
    \xymatrix{&&&z_1 \ar[dr]^{\sigma^{-1} \alpha_1}&&&\\
    x \ar@{~>}[rr]^{\gamma_1}&& z \ar[ur]^{\alpha_1} \ar[dr]_{\alpha_l}& \vdots & \tau^{-1} z \ar@{~>}[rr]^{\gamma_2}&& y\\
    &&&z_l \ar[ur]_{\sigma^{-1} \alpha_l}&&&}
\end{displaymath}

We shall denote by $\Rr(x,y)$ the set of mesh-canonical relations between $x$ and $y$. That way, the space $I(x,y)$ above is the subspace of $k\Ga(x,y)$ generated by $\Rr(x,y)$.

\subsection{Derived mesh-canonical relations}

Now we consider mesh-canonical relations which can be described in terms of those having smaller lengths. 

\begin{defi}

\begin{itemize}

\item If we have a path $\gamma:x \rightsquigarrow z$ and a mesh-canonical relation $r \in \Rr(z,y)$, then $r\gamma$ is a mesh-canonical relation between $x$ and $y$. Then we say that $r\gamma \in \Rr(x,y)$ is a \textbf{mesh-canonical relation derived} from the relation $r \in \Rr(z,y)$. Moreover, we shall also use the notation $\Rr'(z,y) \subseteq \Rr(x,y)$ to denote the subset of $\Rr(x,y)$ given by the mesh-canonical relations between $x$ and $y$ derived from some mesh-canonical relation between $z$ and $y$.

\item If $x$ is a non-injective vertex and $\gamma$ is a path $\tau^{-1} x \rightsquigarrow y$, then $r = \gamma m_x$ is a mesh-canonical relation, originated from the relation of length 2 given by the mesh starting at $x$. We shall denote by $\mathcal{R}_x \subseteq \Rr(x,y)$ the set of all mesh-canonical relations having this form, and the elements of $\mathcal{R}_x$ will be the \textbf{mesh-canonical relations derived from the mesh of $x$}.

\item If $x$ is not injective and $r \in \Rr(\tau^{-1} x, y)$ is a mesh-canonical relation between $\tau^{-1} x$ and $y$, then, for every arrow $\alpha: x \rightarrow x'$ starting at $x$, there is a mesh-canonical relation between $x$ and $y$ given by $r_{\alpha} \doteq r (\sigma^{-1} \alpha)\alpha$. The set $\Rr^d(\tau^{-1} x, y) = \{r_{\alpha}: r \in \Rr(\tau^{-1} x, y), \alpha \in \Delta_1, s(\alpha) = x\}$ is called the set of \textbf{mesh-canonical relations derived from the comesh} of $x$.

\end{itemize}

\begin{ex}
\label{ex:derived mesh-c relations}

Consider the following translation quiver $\Delta$:

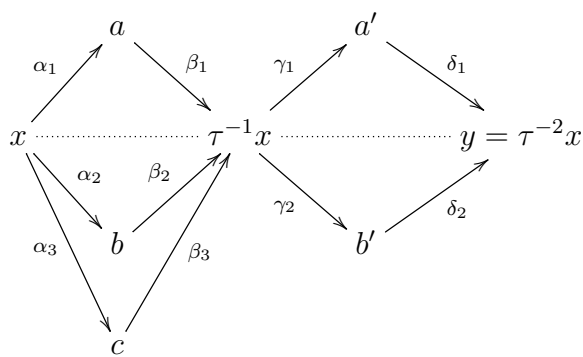
\begin{displaymath}
        \xymatrix{& a \ar[dr]^{\beta_1} && a' \ar[dr]^{\delta_1} & \\
        x \ar[ur]^{\alpha_1} \ar[dr]^{\alpha_2}  \ar[ddr]_{\alpha_3}  \ar@{.}[rr]&& \tau^{-1} x \ar[ur]^{\gamma_1} \ar[dr]_{\gamma_2} \ar@{.}[rr] && y = \tau^{-2} x\\
        & b \ar[ur]^{\beta_2} && b' \ar[ur]_{\delta_2} & \\
        & c \ar[uur]_{\beta_3} &&&}
    \end{displaymath}

The vertices $a,a',b,b',c$ are all projective and injective. 

We have that $r \doteq (\delta_1\gamma_1+\delta_2\gamma_2)\beta_1 \in \Rr(a,y)$ is a mesh-canonical relation between $a$ and $y$. Accordingly, $r' \doteq r \alpha_1 = (\delta_1\gamma_1+\delta_2\gamma_2)\beta_1\alpha_1$ is a mesh-canonical relation between $x$ and $y$, which is derived from $r$. Also, $r'$ is derived from the comesh of $x$. 

Meanwhile, $r'' \doteq \gamma_2(\beta_1\alpha_1 + \beta_2 \alpha_2 + \beta_3\alpha_3)\in \Rr(x,b')$, for instance, is a mesh-canonical relation derived from the mesh of $x$.

\end{ex}

\end{defi}

\section{Gröbner bases for the mesh relations}
\label{sec:grobner bases for mesh}

We are interested to obtain Gröbner bases to the spaces that define the mesh ideal. It occurs that the mesh-canonical relations introduced in the last section will form such a basis (with the addition of a hypothesis) relatively to an order which we call \textit{mesh-lexicographic}.

\subsection{Mesh-lexicographic orders}

\begin{defi}
    Given a path $\gamma: w \rightsquigarrow z$, we say that $\gamma$ \textbf{starts with a path} $\delta:w \rightsquigarrow v$ if there is a path $\gamma':v \rightsquigarrow z$ such that $\gamma = \gamma' \delta$.
\end{defi}

Let $\Delta$ be a translation quiver, and fix a pair of vertices $x,y \in \Delta_0$. 

For each path $\gamma: x \rightsquigarrow z$ with $z$ predecessor of $y$, fix a total order $<_{\gamma}$ in the set of arrows that start at $z$ and finish at some predecessor of $y$.

Then, if $\delta,\delta':x \rightsquigarrow y$ are two distinct paths between $x$ and $y$, there will be $\gamma: x \rightsquigarrow z$ a path and $\alpha:z \rightarrow w$, $\alpha':z \rightarrow w'$ arrows such that $w$ and $w'$ are predecessors of $y$, $\delta$ starts with $\alpha \gamma$ and $\delta'$ starts with $\alpha' \gamma$. Since $\gamma, \alpha, \alpha'$ are uniquely determined by these conditions, the following relation is well-defined:

$$\delta < \delta' \Leftrightarrow \alpha <_{\gamma} \alpha'$$

By the way, this defines a total order $<$ over paths between $x$ and $y$.  We say that such an order is \textbf{lexicographic}.

\begin{defi}
    We say that a total order $<$ on the set of paths between $x$ and $y$ is \textbf{mesh-lexicographic} if it satisfies the following conditions:

    \begin{enumerate}
        \item $<$ is \textit{lexicographic} in the sense explained above.
        
        \item $<$ satisfies the following compatibility condition with the \textit{meshes} of $\Delta$: suppose there is a path $\gamma: x \rightsquigarrow z$ such that $\tau^{-1} z$ exists and is a predecessor of $y$, and let $\alpha: z \rightarrow z'$ be the arrow such that $z'$ is a predecessor of $y$, and such that for every arrow $\alpha': z \rightarrow w$ with $w$ predecessor of $y$ we have $\alpha \geq_{\gamma} \alpha'$. Then for every arrow $\beta: z' \rightarrow v$ such that $v$ is a predecessor of $y$, we have that $\sigma^{-1} \alpha \leq_{\alpha \gamma} \beta$.

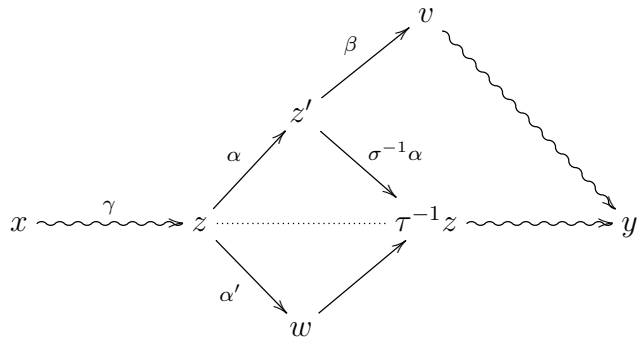
        \begin{displaymath}
            \xymatrix{
            &&&& v \ar@{~>}[ddrr] && \\
            &&& z' \ar[ur]^{\beta} \ar[dr]^{\sigma^{-1} \alpha} &&& \\
            x \ar@{~>}[rr]^{\gamma} && z \ar[ur]^{\alpha} \ar[dr]_{\alpha'} \ar@{.}[rr] && \tau^{-1} z \ar@{~>}[rr] && y \\
            &&& w \ar[ur] &&&}
        \end{displaymath}

        (That means, roughly speaking, that if $\alpha$ is the greatest arrow that starts at $z$, $\sigma^{-1} \alpha$ is the smallest arrow that starts at $z'$).

    \end{enumerate}
\end{defi}

\begin{obs}
    If $<$ is a lexicographic order over the paths between $x$ and $y$, then $<$ induces a total order in the set of arrows that start at $x$ and end at some predecessor of $y$. In fact, that order is $<_{\epsilon_x}$, where $\epsilon_x$ is the zero-length path over the vertex $x$. In the following we will assume that this order is the one we have fixed for the arrows starting at $x$.
\end{obs}

\subsection{On the existence of mesh-lexicographic orders}

\begin{prop}
\label{prop:existe ordem mesh-lex}
    Let $\Delta$ be a translation quiver, and $x,y \in \Delta_0$ be such that all paths between $x$ and $y$ are closed. Let $\beta_1,\ldots,\beta_n$ be all the arrows starting at $x$ and ending at some predecessor of $y$, and suppose we have fixed an order $\beta_1 \preceq \ldots \preceq \beta_n$. Then there is an order $<$ on the paths between $x$ and $y$ satisfying the following properties:

    \begin{enumerate}
        \item $<$ is mesh-lexicographic.

        \item If $\gamma = \alpha_n \ldots \alpha_1: x \rightsquigarrow z$ and $\gamma' = \alpha_n' \ldots \alpha_1': x \rightsquigarrow z$ are two paths such that $z$ is a predecessor of $y$ and for every $1 \leq i \leq n$, the  arrows $\alpha_i$ and $\alpha'_i$ share the same start and end vertices, then $<_{\gamma} = <_{\gamma'}$.

        \item $\beta_1 <_{\epsilon_x} \ldots <_{\epsilon_x} \beta_n$ (i.e., $\preceq$ and $<_{\epsilon_x}$ coincide).
        
    \end{enumerate}
\end{prop}

\begin{proof}
    Since the argument is fairly simple, we only give the basic idea: for every path $\gamma: x \rightsquigarrow z$ where $z$ is a  predecessor of $y$, it is enough to fix an order $<_{\gamma}$ on the arrows that start at $z$ and end at some predecessor of $y$ in such a way that the conditions 1, 2 and 3 from the statement are satisfied. We remark that the hypothesis that all paths between $x$ and $y$ are closed is necessary to guarantee that condition 1 does not conflict with condition 3.
\end{proof}

\subsection{Mesh-canonical relations form Gröbner bases}

Our aim in this section is to state and prove the following theorem, which explains our decision of using Gröbner bases to study the mesh category:

\begin{teo}
\label{th:mesh canonicas eh grobner}
    Let $\Delta$ be a translation quiver and let $x,y \in \Delta_0$ be a pair of vertices such that all paths between $x$ and $y$ are closed. If the paths between $x$ and $y$ are ordered mesh-lexicographically, there is a subset $\Sc(x,y) \subseteq \Rr(x,y)$ satisfying:

    \begin{enumerate}
        \item for every pair of distinct relations $r,s \in \Sc(x,y)$, $\lp(r) \neq \lp(s)$;

        \item for every relation $r \in \Rr(x,y) \setminus \Sc(x,y)$, $r$ is a linear combination of the relations in $\Sc(x,y)$.
    \end{enumerate}
    
    In particular, if $\Rr(x,y)_m$ denotes the set of mesh-canonical relations between $x$ and $y$ with fixed length $m \in \mathbb{N}$, then $\Rr(x,y)_m$ is a Gröbner basis.
\end{teo}

\begin{obs}
    The theory of Gröbner bases also deals with the concepts of \textit{minimal} or \textit{reduced} Gröbner bases (see \cite{AdLous} for those definitions). In general, the Gröbner basis formed by the mesh-canonical relations, as in Theorem~\ref{th:mesh canonicas eh grobner}, is not minimal. The Bimesh Lemma below (Lemma~\ref{lem:bimesh 2}) will imply that some relations are superfluous, i.e., they can be deduced from the others. And even if we discard these superfluous relations in order to get a minimal Gröbner basis, we cannot expect it to be reduced, since usually we can reduce mesh-canonical relations through each other.
\end{obs}

Before proving Theorem~\ref{th:mesh canonicas eh grobner}, we need the following lemma. We have called it the `bimesh lemma' since it deals with interactions between the mesh-canonical relations derived from the mesh starting at $x$ and those derived from the comesh of $x$.

\begin{lem}[The Bimesh Lemma]
\label{lem:bimesh 2}

With the hypotheses of Theorem~\ref{th:mesh canonicas eh grobner}, let $\alpha_1: x \rightarrow x_1$ be the largest arrow of $\Delta$ starting at $x$ and ending at a predecessor of $y$. Define the following set, formed by mesh-canonical relations derived from the mesh:

$$\Sc \doteq \{r \in \Rr_x: \nexists s \in \Rr(\tau^{-1} x, y) \text{ such that }\lp(r) = \lp(s_{\alpha_1})\} \subseteq \Rr_x$$

For $r \in \Rr_x$, we have:

\begin{enumerate}
    \item $r$ is a linear combination of the elements from $\Rr^d(\tau^{-1} x, y) \cup \Sc$.

    \item if $r \in \Sc$, there is no $s \in \Rr(x,y)$ such that $r \neq s$ and $\lp(r) = \lp(s)$.
\end{enumerate}
    
\end{lem}

\begin{proof}
    Note that we can suppose, without loss of generality, that all paths between $x$ and $y$ have the same length. (Indeed, we can build a set $\Sc$ as in the statement for every subset of $\Rr_x$ formed by the relations with a fixed length and then take the union of all those sets).

    That way, using the fact that $\Delta$ is locally finite, there is a finite number of paths between $\tau^{-1} x$ and $y$. Thus let they be $e_1,\ldots,e_N$, listed in such a way that $e_1 (\sigma^{-1} \alpha_1) \alpha_1 < \ldots < e_N (\sigma^{-1} \alpha_1) \alpha_1$. With this notation:

    \begin{itemize}
        \item if $s \in \Rr(\tau^{-1} x , y)$, then there is $I(s) \subseteq \{1,\ldots, N\}$ such that $s = \sum_{i \in I(s)} e_i$, and thus $s_{\alpha_1} = \sum_{i \in I(s)} e_i (\sigma^{-1} \alpha_1) \alpha_1$. Then $\lp(s_{\alpha_1}) = e_{\max I(s)} (\sigma^{-1} \alpha_1) \alpha_1$.

        \item We can denote $\Rr_x = \{\rho_1,\ldots,\rho_N\}$, where for every $1 \leq i \leq N$, $\rho_i = e_i m_x = \sum_{s(\alpha) = x} e_i (\sigma^{-1} \alpha) \alpha$. Also note that $\lp(\rho_i) = e_i (\sigma^{-1} \alpha_1) \alpha_1$.
    \end{itemize}

    Let now $I = \{\max I(s): s \in \Rr(\tau^{-1} x, y)\}$. Then, by construction, for every $1 \leq i \leq N$, $\rho_i \in \Sc$ if and only if $i \notin I$.

    Let us now prove the following fact:

    \textit{Fact.} For every $1 \leq i \leq N$, $e_i$ is a linear combination of the elements of $\Rr(\tau^{-1} x, y) \cup \{e_j : j \notin I\}$.

    Note that, by construction,

    $$\{\lp(r_{\alpha_1}): r \in \Rr(\tau^{-1} x, y)\} \cup \{e_j (\sigma^{-1} \alpha_1) \alpha_1: j \notin I\} = \{e_j (\sigma^{-1} \alpha_1) \alpha_1: 1 \leq j \leq N\}$$

    That is, by definition, $\{r_{\alpha_1}: r \in \Rr(\tau^{-1} x, y)\} \cup \{e_j(\sigma^{-1}\alpha_1)\alpha_1: j \notin I\}$ is a Gröbner basis for the subspace generated by the paths $x \rightsquigarrow y$ which start with $(\sigma^{-1} \alpha_1) \alpha_1$. Therefore, if $1 \leq i \leq N$, there are $s_1,\ldots,s_m \in \Rr(\tau^{-1} x, y)$ and scalars $\lambda_1,\ldots,\lambda_m,\mu_j$ for $j \notin I$ such that
    
    \begin{align*}
    e_i (\sigma^{-1} \alpha_1) \alpha_1 &= \lambda_1 (s_1)_{\alpha_1} + \ldots + \lambda_m (s_m)_{\alpha_1} +  \sum_{j \notin I} \mu_j e_j (\sigma^{-1} \alpha_1) \alpha_1 \\
    &= \lambda_1 s_1  (\sigma^{-1} \alpha_1) \alpha_1 + \ldots + \lambda_m s_m  (\sigma^{-1} \alpha_1) \alpha_1 +  \sum_{j \notin I} \mu_j e_j (\sigma^{-1} \alpha_1) \alpha_1 \\
    &= \left( \lambda_1 s_1  + \ldots + \lambda_m s_m +  \sum_{j \notin I} \mu_j e_j\right) (\sigma^{-1} \alpha_1) \alpha_1
    \end{align*}

Thus $e_i = \lambda_1 s_1  + \ldots + \lambda_m s_m +  \sum_{j \notin I} \mu_j e_j$, proving the fact above.

We now show item 1 in statement: if $1 \leq i \leq N$, then the fact above implies the existence of scalars $\lambda_j ,\mu_j \in k$ and relations $u_1,\ldots,u_q \in \Rr(\tau^{-1} x, y)$ such that 

$$e_i = \lambda_1 u_1 + \ldots + \lambda_q u_q + \sum_{j \notin I} \mu_j e_j$$

so we make

\begin{align*}
   &\rho_i = e_i \left(\sum_{s(\alpha) = x}(\sigma^{-1} \alpha)\alpha \right) \\
   &= \sum_{s(\alpha) = x} \left( \lambda_1 u_1 (\sigma^{-1} \alpha)\alpha + \ldots +\lambda_q u_q (\sigma^{-1} \alpha)\alpha + \sum_{j \notin I} \mu_j e_j (\sigma^{-1} \alpha)\alpha \right)\\
   &= \lambda_1 \left(\sum_{s(\alpha) = x} u_1 (\sigma^{-1} \alpha)\alpha \right) + \ldots +  \lambda_q \left(\sum_{s(\alpha) = x} u_q (\sigma^{-1} \alpha)\alpha\right) + \sum_{j \notin I} \mu_j \rho_j
\end{align*}

Where we have, for every $j$ between 1 and $q$, that $\sum_{s(\alpha) = x} u_j (\sigma^{-1} \alpha)\alpha$ belongs to the set $\Rr^d(\tau^{-1} x, y)$, and each $\rho_j$ with $j \notin I$ belongs to $\Sc$, which shows item 1.

We proceed to the proof of item 2. Fix $r \in \Sc$, and suppose by absurd that there is a $s \in \Rr(x,y)$ different from $r$ and such that $\lp(r) = \lp(s)$. Since $r \in \Sc \subseteq \Rr_x$, there is $i$ between 1 and $N$ such that $r = \rho_i$. Since $\lp(\rho_j) = e_j (\sigma^{-1} \alpha_1) \alpha_1$, there is no way we have $\lp(\rho_i) = \lp(\rho_j)$ if $i \neq j$. That proves $s \notin \Rr_x$. So there are a non-injective vertex $z$ and paths $\gamma: x \rightsquigarrow z$ and $\delta: \tau^{-1} z \rightsquigarrow y$ such that $s = \delta m_z \gamma$ and $\gamma$ has length greater than or equal to 1.

Let $\beta_1: z \rightarrow z_1, \ldots, \beta_t: z \rightarrow z_t$ be the arrows that start at $z$, listed in such a way that $\delta (\sigma^{-1} \beta_1) \beta_1 \gamma > \ldots > \delta (\sigma^{-1} \beta_t) \beta_t \gamma$. In particular $\lp(s) = \lp\left(\delta \left(\sum_{j=1}^t (\sigma^{-1} \beta_j) \beta_j \right)\gamma\right) = \delta (\sigma^{-1} \beta_1) \beta_1 \gamma$.

We have two cases to analyze:

\begin{itemize}
    \item If $\gamma$ has length 1, then $\gamma$ is just a single arrow $\alpha: x \rightarrow z$. If $\alpha \neq \alpha_1$, then $\alpha < \alpha_1$, because $\alpha_1$ is by definition the largest arrow that starts at $x$, and so $\delta (\sigma^{-1} \beta_1) \beta_1 \alpha < e_i (\sigma^{-1} \alpha_1) \alpha_1$, that is, $\lp(s) < \lp(r)$, contradiction. Then we must have $\alpha = \alpha_1$. Since $<$ is mesh-lexicographic and $\alpha_1$ is the largest arrow that starts at $x$, the paths that start with $(\sigma^{-1} \alpha_1) \alpha_1$ must be smaller than those which start with $\alpha_1$ but not with $(\sigma^{-1} \alpha_1) \alpha_1$. But then the only possibility for us to have $\lp(s) = \delta (\sigma^{-1} \beta_1) \beta_1 \alpha_1 = e_i (\sigma^{-1} \alpha_1) \alpha_1$ is if $\sigma^{-1} \alpha_1 = \beta_1$ were the only arrow starting at $z$, and from that we would have $x \xrightarrow{\alpha_1} z \xrightarrow{\beta_1} z_1 \xrightarrow{\sigma^{-1}\beta_1} \tau^{-1} z \rightsquigarrow_{\delta} y$ being a open path between $x$ and $y$, going against the hypothesis. So we can move on to the next case.

    \item Suppose $\gamma$ has length at least 2, and therefore that it starts with $\beta \alpha$, where $\alpha:x \rightarrow u$ and $\beta:u \rightarrow v$ are arrows. Since $\lp(s) = \delta (\sigma^{-1} \beta_1) \beta_1 \gamma= e_i (\sigma^{-1} \alpha_1) \alpha_1 = \lp(r)$, we already obtain that $\alpha = \alpha_1$ and $\beta = \sigma^{-1} \alpha_1$. But that means that $s \in \Rr^d(\tau^{-1} x, y)$, and that $s$ has the form $s=s'_{\alpha_1}$, where $s' \in \Rr(\tau^{-1} x, y)$. Since $r \in \Sc$, the fact that we have $\lp(r) = \lp(s)$ therefore contradicts the way we have defined $\Sc$, finishing this case and the proof.
\end{itemize}
\end{proof}

\begin{ex}
    We go back to the Example~\ref{ex:derived mesh-c relations}, in order to illustrate the Bimesh Lemma above. 

    
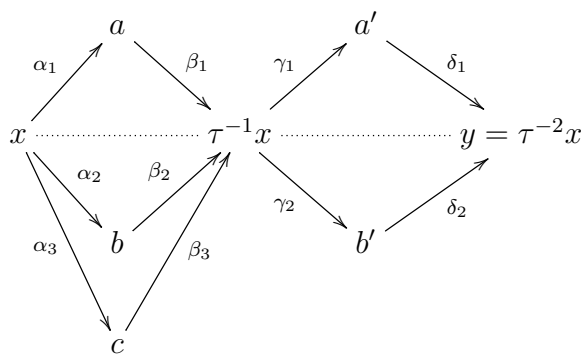
\begin{displaymath}
        \xymatrix{& a \ar[dr]^{\beta_1} && a' \ar[dr]^{\delta_1} & \\
        x \ar[ur]^{\alpha_1} \ar[dr]^{\alpha_2}  \ar[ddr]_{\alpha_3}  \ar@{.}[rr]&& \tau^{-1} x \ar[ur]^{\gamma_1} \ar[dr]_{\gamma_2} \ar@{.}[rr] && y = \tau^{-2} x\\
        & b \ar[ur]^{\beta_2} && b' \ar[ur]_{\delta_2} & \\
        & c \ar[uur]_{\beta_3} &&&}
    \end{displaymath}

Here we have 6 paths from $x$ to $y$, which we denote \begin{align*}
    e_1 = \delta_1 \gamma_1 \beta_1 \alpha_1, & \hspace{1cm} e_3 = \delta_1 \gamma_1 \beta_2 \alpha_2,  \hspace{1cm} e_5 = \delta_1 \gamma_1 \beta_3 \alpha_3, \\
    e_2 = \delta_2 \gamma_2 \beta_1 \alpha_1, & \hspace{1cm} e_4 = \delta_2 \gamma_2 \beta_2 \alpha_2, \hspace{1cm} e_6 = \delta_2 \gamma_2 \beta_3 \alpha_3
\end{align*}

And note that we have named them in such a way that the natural order $e_1 > e_2 > e_3 > e_4 > e_5 > e_6$ is mesh-lexicographic. We have 5 mesh-canonical relations between $x$ and $y$, which we denote by

\begin{align*}
    r_1 = e_1 + e_2 & \hspace{1cm} r_4 = e_1 + e_3 + e_5\\
    r_2 = e_3 + e_4 & \hspace{1cm} r_5 = e_2+ e_4 + e_6\\
    r_3 = e_5 + e_6 & 
\end{align*}

The mesh-canonical relations derived from the mesh of $x$ are given by $\Rr_x = \{r_4, r_5\}$, while the relations derived from the comesh of $x$ form the set $\Rr^d(\tau^{-1}x, y) = \{r_1,r_2,r_3\}$.

As it comes to the leading basis elements, we have $\lp(r_1) = e_1, \lp(r_2) = e_3, \lp(r_3) = e_5$, while $\lp(r_4) = e_1$ (thus the same as $r_1$), and $\lp(r_5) = e_2$. Therefore, by the definition of $\Sc$ above, we have $\Sc = \{r_5\}$, when the Bimesh Lemma tells us that the none of $r_1,r_2,r_3,r_4$ have the same leading basis element as $r_5$ (which we have just shown again), and that every element in $\Rr_x = \{r_4,r_5\}$ is a linear combination of $\Rr^d(\tau^{-1} x, y) \cup \Sc = \{r_1,r_2,r_3,r_5\}$.  In fact, $r_5 \in \Sc$ is a linear combination of itself, and $r_4$ is a linear combination of $r_1,r_2,r_3,r_5$, as follows:

\begin{align*}
    r_4 &= e_1 + e_3 + e_5 = (e_1+e_2)-e_2 + (e_3+e_4)-e_4 + (e_5+e_6)-e_6 \\
    &= (e_1+e_2) + (e_3+e_4) + (e_5+e_6) - (e_2+e_4+e_6) = r_1 + r_2 + r_3 - r_5
\end{align*}

\end{ex}

\begin{proof}[Proof of Theorem~\ref{th:mesh canonicas eh grobner}]
Note that the second assertion in the statement of this theorem, which says that items 1 and 2 in particular imply that $\Rr(x,y)_m$ is a Gröbner basis, follows by an easy argument from the definition of Gröbner bases. So our work is to show items 1 and 2.

Once more, note that we can suppose without loss of generality that the paths $x \rightsquigarrow y$ all have the same length $m$. As so, we can prove the theorem by induction on the length $m$ of these paths.

If that length is 1, then it is trivial, because the mesh-canonical relations only involve paths of length at least 2, which implies $\Rr(x,y) = \emptyset$, and thus the statement is easily verified in this case. Let us see the case where the length is greater than 1.

Let $\alpha_1: x \rightarrow x_1, \ldots, \alpha_t: x \rightarrow x_t$ be the arrows that start at $x$ and end at some predecessor of $y$, and suppose $\alpha_1 > \ldots > \alpha_t$. Note that we have a disjoint union:

    $$\Rr(x,y) = \Rr_x \sqcup \Rr'(x_1,y) \sqcup \ldots \sqcup \Rr'(x_t,y)$$

where $\Rr_x$ is the set of relations derived from the mesh that starts at $x$ (where it could be that $\Rr_x = \emptyset$), and for each $i$, $\Rr'(x_i,y)$ denotes the set of mesh-canonical relations derived from the ones between $x_i$ and $y$.

Now, we can assume, using the induction hypothesis, that for each $1 \leq i \leq t$, there is a subset $\Sc(x_i,y) \subseteq \Rr'(x_i,y)$ such that:

    \begin{enumerate}
        \item[(1a)] $\lp(r) \neq \lp(s)$ for every pair $r,s \in \Sc(x_i,y)$ with $r \neq s$;

        \item[(2a)] for every relation $r \in \Rr'(x_i,y) \setminus \Sc(x_i,y)$, $r$ is a linear combination of the elements from $\Sc(x_i,y)$.
    \end{enumerate}

On the other hand, by the Bimesh Lemma, there is a subset $\Sc \subseteq \Rr_x$ such that:

\begin{enumerate}
    \item[(1b)] If $r \in \Sc$, there is no $s \in \Rr'(x_1,y) \cup \ldots \cup \Rr'(x_t,y)$ such that $\lp(r) = \lp(s)$;
    
    \item[(2b)] If $r \in \Rr_x \setminus \Sc$, $r$ is a linear combination of the elements from $\Sc \cup \Rr'(x_1,y) \cup \ldots \cup \Rr'(x_t,y)$.
\end{enumerate}

Define then $\Sc(x,y) \doteq \Sc \cup \Sc(x_1,y) \cup \ldots \cup \Sc(x_t,y)$. Let us show that $\Sc(x,y)$ verifies items 1 and 2 from the statement:

\begin{enumerate}
    \item Fix two relations $r,s \in \Sc(x,y)$, with $r \neq s$. Let us see that $\lp(r) \neq \lp(s)$, which proves item 1 in the statement. We have four cases:

    \begin{itemize}
        \item $r,s \in \Sc \subseteq \Rr_x$:  then $\Rr_x \neq \emptyset$, and by the definition of $\Rr_x$, there are distinct paths $\gamma,\delta: \tau^{-1} x \rightsquigarrow y$ such that $r = \gamma m_x$ and $s = \delta m_x$. Then, since $\alpha_1$ is the largest of the arrows that start at $x$ and the order $<$ is lexicographic, we have $\lp(r) = \gamma (\sigma^{-1} \alpha_1) \alpha_1 \neq \delta (\sigma^{-1} \alpha_1) \alpha_1 = \lp(s)$.

        \item There is an $i$ such that $r,s \in \Sc(x_i,y)$: then we use use (1a) to obtain that $\lp(r) \neq \lp(s)$.

        \item There are $i,j$ with $i \neq j$ such that $r \in \Sc(x_i,y)$ and $s \in \Sc(x_j,y)$: in this case $\lp(r)$ is a path that starts with $\alpha_i$ and $\lp(s)$ is a path that starts with $\alpha_j$. In particular, $\lp(r) \neq \lp(s)$.

        \item $r \in \Sc$ and $s \in \Sc(x_i,y)$ for some $1 \leq i \leq t$ (or the contrary): it suffices to use (1b).
    \end{itemize}
    \item If $r \in \Rr(x,y) \setminus \Sc(x,y)$, then we have two possibilities: first, suppose $r \in \Rr_x \setminus \Sc$. Then, by (2b), $r$ is a linear combination of $\Sc \cup \Rr'(x_1,y) \cup \ldots \cup \Rr'(x_t,y)$, with each $\Rr'(x_i,y)$ being generated, using (2a), by $S(x_i,y)$. Then $r$ is a linear combination of $\Sc \cup \Sc(x_1,y) \cup \ldots \cup \Sc(x_t,y) = \Sc(x,y)$. The second possibility is to have $r \in \Rr'(x_i,y) \setminus \Sc(x_i,y)$. But then, by (2a), $r$ is a linear combination of $\Sc(x_i,y) \subseteq \Sc(x,y)$. That proves item 2.
\end{enumerate}
\end{proof}

\begin{ex}
    We shall give an example to show that the hypothesis of all paths between $x$ and $y$ being closed is necessary for the Theorem~\ref{th:mesh canonicas eh grobner} to hold. Let $\Delta$ be the following translation quiver:

    
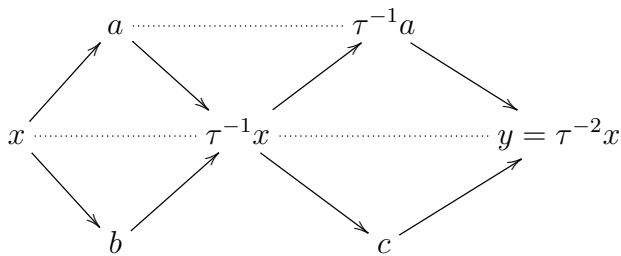
\begin{displaymath}
        \xymatrix{& a \ar[dr] \ar@{.}[rr]&& \tau^{-1} a \ar[dr]& \\
        x \ar[ur] \ar[dr] \ar@{.}[rr]&& \tau^{-1} x \ar[ur] \ar[dr] \ar@{.}[rr] && y = \tau^{-2} x\\
        & b \ar[ur]&& c \ar[ur]&}
    \end{displaymath}

The vertices $x$ and $a$ are projective, the vertices $\tau^{-1} a$ and $y$ are injective, while the vertices $b$ and $c$ are both projective and injective.

We have 4 paths between $x$ and $y$, which we enumerate in the following way:

\begin{align*}
    e_1: & \text{ the path which passes through } x,a,\tau^{-1} x, \tau^{-1} a, y; \\
    e_2: & \text{ the path which passes through } x,a,\tau^{-1} x, c, y; \\
    e_3: & \text{ the path which passes through } x,b,\tau^{-1} x, \tau^{-1} a, y; \text{ and }\\
    e_4: & \text{ the path which passes through } x,b,\tau^{-1} x, c, y
\end{align*}

The paths between $x$ and $y$ are not all closed, because $e_1$ is an open path.

Over these paths, we fix the following order, which is mesh-lexicographic: $e_2 > e_1 > e_3 > e_4$. We have 5 mesh-canonical relations, written below with descending order from terms to the left to the right:

\begin{align*}
    r_1 = e_2 + e_1 & \hspace{2cm} r_3 = e_1 & r_4 = e_2 + e_4 \\
    r_2 = e_3 + e_4 && r_5 = e_1 + e_3
\end{align*}

Then observe that $\Rr(x,y) = \{r_1,r_2,r_3,r_4,r_5\}$ is not a Gröbner basis: for instance, $e_4 = (e_3+e_4)-(e_1+e_3)+e_1 = r_2 - r_5 + r_3$ belongs to the subspace generated by $\Rr(x,y)$, but there is no relation $r_i$ in $\Rr(x,y)$ whose leading term $\lp(r_i)$ is $e_4$. 
 
\end{ex}

\section{Zero paths in the mesh category}
\label{sec:zero paths}

It is time to use the Gröbner basis given by Theorem~\ref{th:mesh canonicas eh grobner} to derive the following result, which asserts the non-existence of zero combinations of paths in the mesh category between two vertices if the paths between these vertices are all closed. This will be a key step towards proving our main theorem.

\begin{teo}
\label{th:his neq 0}
    Let $\Delta$ be a translation quiver, and let $x,y \in \Delta_0$ be such that all the paths between $x$ and $y$ are pairwise homotopic and closed. Suppose there is a path $x=x_0 \rightarrow x_1 \rightarrow \ldots \rightarrow x_n = y$ in $\Delta$. Let, for all $1 \leq i \leq n$, $\alpha_{i1},\ldots,\alpha_{id_i}$ be all the arrows between $x_{i-1}$ and $x_i$ in $\Delta$, and let $a_{ij} \in k$, $1 \leq i \leq n, 1 \leq j \leq d_i$ be scalars in $k$ such that for every $i$ there is a $j$ such that $a_{ij} \neq 0$. Define, for every $1 \leq i \leq n$, $h_i = a_{i1} \alpha_{i1} + \ldots + a_{id_i} \alpha_{id_i} \in k\Delta(x_{i-1},x_i)$. Then $\overline{h_n} \ldots \overline{h_1} \neq 0$ in $k(\Delta)$.
\end{teo}

\begin{proof}

Since all paths between $x$ and $y$ are pairwise homotopic, they have in particular the same length. In order to show that $\overline{h_n} \ldots \overline{h_1} \neq 0$ in $k(\Delta)$, we have to show that $h_n \ldots h_1 \in k\Delta(x,y)$ cannot be reduced to zero via the mesh-canonical relations, because $k(\Delta)(x,y)$ is the quotient of the space $k\Delta(x,y)$ by the subspace generated by the mesh-canonical relations between $x$ and $y$, which form a Gröbner basis according to Theorem~\ref{th:mesh canonicas eh grobner}.

The proof is then done by induction on $n$. If $n=1$, the proof is trivial because all paths between $x$ and $y$ will have length one and thus there will be no mesh-canonical relations between these vertices (all such relations should have length at least two). Thus $\overline{h_1} \neq 0$. 

Suppose $n> 1$. Let $\alpha_1,\ldots,\alpha_q: x \rightarrow x_1$ be all arrows from $x$ to $x_1$, and define $a_1 = a_{11}, a_2 = a_{12}, \ldots, a_q = a_{1q}$, in such a way that we have $h_1 = a_1 \alpha_1 + \ldots + a_q \alpha_q$ and with at least one of $a_1,\ldots,a_q$ being non-zero. Without loss of generality, we may suppose $a_q \neq 0$.

Using Proposition~\ref{prop:existe ordem mesh-lex}, there is a mesh-lexicographic order $<$ on the paths between $x$ and $y$ satisfying the conditions from that proposition and such that:
\begin{enumerate}
    \item[a)] $\alpha_1 >_{\epsilon_x} \ldots >_{\epsilon_x} \alpha_q$
    \item[b)] $\beta > \alpha_i$, where $1 \leq i \leq q$ and $\beta$ is any possible arrow starting at $x$ and not ending in $x_1$.
\end{enumerate}

By the induction hypothesis, $\overline{h_n} \ldots \overline{h_2} \neq 0$, and therefore, if $\delta_1,\ldots,\delta_N: x_1 \rightsquigarrow y$ are all the paths between $x_1$ and $y$, there are $\mu_1,\ldots,\mu_N \in k$ with not all being zero such that $h_n \ldots h_2  = \mu_1 \delta_1 + \ldots + \mu_N \delta_N$. Therefore:

$$h_n \ldots h_2 h_1 = \sum_{i=1}^N \sum_{j=1}^q a_j \mu_i \delta_i \alpha_j$$

Also because $\overline{h_n} \ldots \overline{h_2} \neq 0$, there will be scalars $\lambda_1,\ldots,\lambda_N \in k$ with not all being zero such that $\sum_{i=1}^N \mu_i \delta_i \alpha_1 \xrightarrow{\Rr'(x_1,y)}_+ \sum_{i=1}^N \lambda_i \delta_i \alpha_1$, with $\sum_{i=1}^N \lambda_i \delta_i \alpha_1$ being reduced relatively to $\Rr'(x_1,y)$, which is the set of mesh-canonical relations derived from the ones between $x_1$ and $y$.

Since $<$ is mesh-lexicographic, we can conclude that for $1 \leq j \leq q$, $\sum_{i=1}^N \mu_i \delta_i \alpha_j \xrightarrow{\Rr'(x_1,y)}_+ \sum_{i=1}^N \lambda_i \delta_i \alpha_j$, with $\sum_{i=1}^N \lambda_i \delta_i \alpha_j$ being reduced relatively to $\Rr'(x_1,y)$ (i.e., the same reduction which could be carried when $\alpha_1$ was placed ahead also works when we place another arrow $\alpha_j$).

From that, we have

 $$h_n \ldots h_2 h_1 = \sum_{i=1}^N \sum_{j=1}^q a_j \mu_i \delta_i \alpha_j \xrightarrow{\Rr'(x_1,y)}_+ \sum_{i=1}^N \sum_{j=1}^q a_j \lambda_i \delta_i \alpha_j$$
 
with $\sum_{i=1}^N \sum_{j=1}^q a_j \lambda_i \delta_i \alpha_j$ being reduced relatively to $\Rr'(x_1,y)$. To save notation, let us denote by $\rho = \sum_{i=1}^N \sum_{j=1}^q a_j \lambda_i \delta_i \alpha_j$ the remainder of this reduction.

If $r \in \Rr'(z,y)$, where $z \in x^+$ and $z \neq x_1$, then by the construction of the order $<$, $\lp(r) > \delta_i \alpha_j$ for all $1 \leq i \leq N$ and $1 \leq j \leq q$. That shows $\rho$ is reduced relative to $\Rr'(z,y)$.

Since $\Rr(x,y) = \Rr'(x_1,y) \sqcup \bigsqcup_{\substack{z \in x^+ \\ z \neq x_1}} \Rr'(z,y) \sqcup \Rr_x$, it remains to analyze the reductions of $\rho$ modulo the relations in $\Rr_x$. We have three cases:

\begin{enumerate}
    \item If $\Rr_x = \emptyset$, then $\rho$ is reduced relatively to $\Rr_x$, and therefore relatively to $\Rr(x,y)$. This means we have shown that $h_n \ldots h_2 h_1 \xrightarrow{\Rr(x,y)}_+ \rho \neq 0$, where $\rho$ is reduced relative to $\Rr(x,y)$, as we wanted to prove.
    
    \item If $\Rr_x \neq \emptyset$ and there is $z \in x^+ \setminus \{x_1\}$, then, using b) above, $\lp(r) > \delta_i \alpha_j$ for every relation $r \in \Rr_x$ and for all $i,j$. Thus $\rho$ is reduced relatively to $\Rr_x$ and the conclusion is identical to the one in the previous case.

    \item The most demanding case is the one where $\Rr_x \neq \emptyset$ and $x^+ = \{x_1\}$.

    Let $\gamma_1,\ldots,\gamma_p$ be all paths from $\tau^{-1} x$ to $y$. If we define, for each $1 \leq l \leq p$, $r_l = \gamma_l \left( \sum_{j=1}^q (\sigma^{-1} \alpha_j) \alpha_j \right)$, we will have that $\Rr_x = \{r_1,\ldots,r_p\}$.

    Note that $N = qp$, and that we can, using the hypothesis on $<$, suppose that $\delta_i = \gamma_l (\sigma^{-1} \alpha_j)$, where $i = (j-1)p+l$. So:

    \begin{align*}
        \rho = \sum_{i=1}^N \sum_{j=1}^q a_j \lambda_i \delta_i \alpha_j &= \sum_{l=1}^p \sum_{j'=1}^q \sum_{j=1}^q a_j \lambda_{(j'-1)p+l} \gamma_l (\sigma^{-1} \alpha_{j'}) \alpha_j \\
        \xrightarrow{\{r_1,\ldots,r_p\}}_+ &\left(\sum_{l=1}^p \sum_{\substack{1 \leq j,j' \leq q \\ j \neq j'}} a_j \lambda_{(j'-1)p+l} \gamma_l (\sigma^{-1} \alpha_{j'}) \alpha_j \right. \\
        &\left. + \sum_{l=1}^p \sum_{j=2}^q (a_j \lambda_{(j-1)p+l} - a_1 \lambda_l)\gamma_l (\sigma^{-1} \alpha_{j}) \alpha_j \right) \doteq \rho_0
    \end{align*}

and note that $\rho_0$ is reduced relative to $\{r_1,\ldots,r_p\}$, in fact because for every $1 \leq l \leq p$, $\lp(r_l) = \gamma_l(\sigma^{-1} \alpha_1)\alpha_1$, since $<$ is mesh-lexicographic and $\alpha_1 > \ldots > \alpha_q$.

Suppose by absurd that $\rho_0 = 0$. Then, since $\rho_0$ is reduced, we can conclude that $a_j \lambda_{(j'-1)p+l} = 0$ for every $1 \leq l \leq p$, $1 \leq j,j' \leq q$ with $j \neq j'$, and $a_j \lambda_{(j-1)p+l} - a_1 \lambda_l = 0$ for every $1 \leq l \leq p$, $2 \leq j \leq q$. In particular, putting $j= q$, we have a system

$$\begin{cases}
    a_q \lambda_{(j'-1)p+l} = 0 & \text{ for all } 1 \leq l \leq p \text{ and } 1 \leq j' \leq q-1\\
    a_q \lambda_{(q-1)p+l} - a_1 \lambda_l = 0 & \text{ for all } 1 \leq l \leq p
\end{cases}$$

or, equivalently,

$$\begin{cases}
    a_q \lambda_i = 0 & \text{, if } 1 \leq i \leq (q-1)p\\
    a_q \lambda_i - a_1 \lambda_{i-(q-1)p} = 0 & \text{, if } (q-1)p < i \leq qp=N
\end{cases}$$

which gives us the following homogeneous linear system whose unknowns are the $\lambda_i$'s:

\begin{displaymath}
    \begin{bmatrix}
        a_q &&&&& \\
        0 && \ddots &&& 0 & \\
        &&& a_q &&& \\
        -a_1 && \ddots && a_q && \\
        & \ddots &&&& \ddots & \\
        0 && -a_1 && 0 && a_q
    \end{bmatrix} \begin{bmatrix}
        \lambda_1 \\
        \vdots \\
        \lambda_{(q-1)p} \\
        \lambda_{(q-1)p+1} \\
        \vdots \\
        \lambda_N
    \end{bmatrix} = 0
\end{displaymath}

The determinant of the matrix that originates the system is equal to $a_q^N$, which is non-zero because we have supposed right from the beginning that $a_q \neq 0$. That shows that the system above is determinate and so that $\lambda_1 = \ldots = \lambda_N = 0$, a contradiction because we had already observed that at least one of the $\lambda_i$'s is non-zero. The absurd happened because we have supposed that  $\rho_0 = 0$. Thus $\rho_0 \neq 0$ and that concludes the proof of Theorem~\ref{th:his neq 0}.

\end{enumerate}

\end{proof}

\section{Proof of the Main Theorem}
\label{sec:proof of main thm}

We now give the proof of the main theorem, stated in the introduction.

\begin{proof}[Proof of the Main Theorem]
Clearly, statement 1 is the contrapositive of statement 2, so it is enough to show the former.

The hypothesis from statement 1 implies that all paths between $X$ and $Y$ are closed.  Consider a covering $\pi: \Delta \rightarrow \Gamma$ and let $x=x_0 \rightarrow x_1 \rightarrow \ldots \rightarrow x_n = y$ be a path in $\Delta$ whose image under $\pi$ is the path $X=X_0 \xrightarrow{h_1} X_1 \xrightarrow{h_2} \ldots \xrightarrow{h_n} X_n = Y$ from the statement. Using the properties of coverings, we obtain that all paths between $x$ and $y$ in $\Delta$ are closed too.

Let, for every $1 \leq i \leq n$, $\alpha_{i1},\ldots,\alpha_{id_i}$ be all the arrows between $X_{i-1}$ and $X_i$. By definition of the Auslander-Reiten quiver, and since $k$ is algebraically closed, the set $\{\alpha_{i1},\ldots,\alpha_{id_i}\}$ can be assumed to consist of irreducible morphisms whose classes modulo $\rad^2$ form a $k$-basis of $\rad(X_{i-1},X_i)/\rad^2(X_{i-1},X_i)$, so that there will be scalars $a_{i1},\ldots,a_{id_i} \in k$, with not all of them being zero, such that $\overline{h_i} = a_{i1} \overline{\alpha_{i1}} + \ldots + a_{id_i} \overline{\alpha_{id_i}}$.

In the proof of \cite{CMT2}, Proposition 3 (see also \cite{CCs}, Corollary 5.1.5, where we have explained this with more details), it is proved that if $\Delta$ were the {\it universal covering} of $\Gamma$, and we had, by absurd, that $h_n \ldots h_1 \in \rad^{n+1}(X,Y)$, then this would imply $\overline{h_n} \ldots \overline{h_1} = 0$, where the $\overline{h_i}$'s can be seen (by lifting) as morphisms in the mesh category $k(\Delta)$.

We now want to apply Theorem~\ref{th:his neq 0}, and for that we use that all paths between $x$ and $y$ are closed, and that since mesh-canonical relations are linear combinations of paths sharing the same homotopy class, we can suppose without loss of generality that all the paths between $X$ and $Y$ are pairwise homotopic (and so, that will also be the case for the paths between $x$ and $y$, by the properties of coverings). Then, from Theorem~\ref{th:his neq 0}, we obtain exactly that $\overline{h_n} \ldots \overline{h_1} \neq 0$, a contradiction which proves the theorem.

\end{proof}

\begin{obs}
Observe that Theorem 2.2 from \cite{CCT1} says (through the implications $(a) \Rightarrow (b), (c), (d)$ there stated) that if $h: X \rightarrow Y$ and $h': Y \rightarrow Z$ are irreducible morphisms such that $h'h \in \rad^3(X,Z)$, then there is an Auslander-Reiten sequence $0 \rightarrow X \xrightarrow{f} Y \xrightarrow{g} Z \rightarrow 0$, whose middle term $Y$ is indecomposable. This fact can also be derived from the main theorem here. We leave it to the reader to verify that our main theorem also appears implicitly in other results in literature regarding compositions of irreducible morphisms in particular cases, such as in Lemma 2.3 from \cite{CCT2} and the theorem from \cite{CCT3}. 

\end{obs}


\section*{Acknowledgments}

This work is part of the PhD thesis (\cite{CCt}) of the first named author, under supervision by the second named author. The authors gratefully acknowledge financial support by São Paulo Research Foundation (FAPESP), grants \#2020/13925-6, \#2022/02403-4 and by CNPq (grant Pq 312590/2020-2).


\begin{thebibliography}{9999}

\bibitem[{\bf AL}]{AdLous}
W. Adams, P. Loustaunau, {\it An Introduction to Gröbner Bases}, Graduate Studies in Mathematics {\bf 3}, AMS 1994. 

\bibitem[{\bf ASS}]{ASS}
I. Assem, D. Simson, A. Skowro\'{n}ski, {\em Elements of the representation theory of associative algebras}, London Math. Soc. Students Texts {\bf 65}, Cambridge University Press, 2006. 

\bibitem[{\bf AC}]{AC}
I. Assem, F. U. Coelho, {\it Basic Representation Theory of Algebras}, Graduate Texts in Mathematics {\bf 283}, 
Springer 2020. 

\bibitem[{\bf AC2}]{AC2}
I. Assem, F. U. Coelho, {\it An introduction to module theory}, Oxford Graduate Texts in Mathematics {\bf 32}, 
Oxford Univ. Press, 2024, 608 pp. 

\bibitem [{\bf ARS}]{ARS}
M.\ Auslander, I.\ Reiten and S.\ Smal\o , {\it  Representation theory of Artin  algebras}, Cambridge Studies in Advanced Mathematics {\bf 36}, Cambridge Univ. Press, 1995.

\bibitem [{\bf Berg}]{Berg} 
G. Bergman {\it The diamond lemma for ring theory}, Invent. Math. {\bf 29} (1978), 178-218.

\bibitem [{\bf BG}]{BG} 
K. Bongartz, P. Gabriel, {\it Covering spaces in representation theory}, Invent. Math. {\bf 65} (1982), 331-378.

\bibitem [{\bf Buch}]{Buch} 
B. Buchberger {\it Ein Algorithmus zum Auffinden der Basiselemente des Restklassenringes nach einem nulldimensionalen Polynomideal}, PhD Thesis, Univ. of Innsbruck, Austria (1965).

\bibitem [{\bf CCT1}]{CCT1}
C. Chaio, F. U. Coelho, S. Trepode, {\it On the composite of two irreducible morphisms in radical cube}, J. Algebra {\bf 312} (2007), 650-667.

\bibitem [{\bf CCT2}]{CCT2}
C. Chaio, F. U. Coelho, S. Trepode, {\it On the composite of  irreducible morphisms in almost sectional paths}, J. Pure Applied Algebra {\bf 212} (2008), 244-261. 

\bibitem [{\bf CCT3}]{CCT3}
C. Chaio, F. U. Coelho, S. Trepode, {\it On the composite of  three  irreducible morphisms  in the fourth power of the radical}, Comm. Algebra {\bf 39},2 (2011), 555-559. 

\bibitem [{\bf CMT1}]{CMT1}
C. Chaio,  P. Le Meur, S. Trepode, {\it Degrees of irreducible morphisms and finite representation 
type}, J. London Math. Soc. {\bf 84} (2011) 35-57.

\bibitem [{\bf CMT2}]{CMT2}
C. Chaio,  P. Le Meur, S. Trepode, {\it Covering techniques in Auslander-Reiten theory},  J. Pure Appl. Algebra {\bf 223} (2019), 641-659. 

\bibitem [{\bf CPT}]{CPT}
C. Chaio, M. I. Platzeck, S. Trepode, {\it The composite of irreducible morphisms in regular components},  Colloq. Math. {\bf 123}(1) (2011),  27-47.

\bibitem [{\bf CT}] {CT}
C. Chaio, S. Trepode, {\it The composite of irreducible morphisms in standard components}, J. Algebra {\bf 323} (2010) 1000-1011.

\bibitem[{\bf C-th}]{CCt}
V. Chust, {\it Riedtmann functors and compositions of irreducible morphisms}, PhD Thesis (Advisor: F. U. Coelho), Inst. of Mathematics and Statistics, Univ. of São Paulo, Brasil, 2025. 117 pp.

\bibitem[{\bf CC}]{CCs}
V. Chust, F. U. Coelho, {\it On Riedtmann's well-behaved functors and applications to composites of irreducible morphisms}, preprint (arXiv 2507.03121), 2025.

\bibitem[{\bf CLS}]{CLS}
D. Cox, J. Little, D. O'Shea, {\it Ideals, Varieties and Algorithms: An Introduction to Computational Algebraic Geometry and Commutative Algebra}, 4th ed., Undergraduate Texts in Mathematics, Springer, 2015.

\bibitem [{\bf CHR}]{CHR}
W. Crawley-Boevey, D. Happel, C. M. Ringel, {\it A bypass of an arrow is sectional},  Arch. Math. {\bf 58} (1992),  525-528.

\bibitem [{\bf FFG}]{FFG}
D. Farkas, C. Feustel, E. Green, {\it Synergy in the theories of Gröbner bases and path algebras},  Can. J. Math. {\bf 45}(4) (1993),  727-739.

\bibitem [{\bf Gre}]{Gre}
E. Green, {\it Multiplicative bases, Gröbner bases, and right Gröbner bases},  J. Symbolic Computation {\bf 29} (2000), 601-623.

\bibitem [{\bf IT}]{IT1}
K. Igusa, G. Todorov, {\it A characterization of finite Auslander-Reiten quivers}, J. Algebra {\bf 89}(1984), 148-177.

\bibitem [{\bf LM}]{LM}
P. Le Meur, {\it The universal cover of an algebra without double bypass}, J. Algebra {\bf 312} (2007), 330-353.

\bibitem [{\bf Liu1}]{Liu1}
S. Liu, {\it The degree of irreducible maps and the shapes of Auslander-Reiten quivers}, J. London Math. Soc. {\bf 45-2} (1992), 32-54.

\bibitem [{\bf Liu2}]{Liu2}
S. Liu, {\it Infinite radicals in standard Auslander-Reiten components}, J. Algebra {\bf 166} (1994), 245-254.

\bibitem [{\bf Rie}]{Rie}
C. Riedtmann, {\it Algebren, darstellungsk\"ocher, \"uberlagerungen und z\"uruck}, Comment. Math. Helv. 55 (1980), 199-224.

\bibitem [{\bf Rin}]{Rin}
 C.~ M.~ Ringel,  Tame algebras and integral quadratic forms, Springer Lecture Notes in Maths {\bf 1099} (1984).
\end{thebibliography}
\end{document}